\documentclass[11pt,letterpaper,notitlepage,twoside]{article}

\usepackage{titling}
\usepackage{enumitem}

\usepackage[top=4cm, bottom=3.5cm, left=3.8cm, right=3.8cm, columnsep=20pt]{geometry}
\usepackage{indentfirst}
\usepackage[dvipsnames]{xcolor}
\definecolor{FrenchPink}{RGB}{255,118,164}
\definecolor{VividMulberry}{RGB}{192,17,215}
\definecolor{DenimBlue}{RGB}{47,60,190}
\usepackage{hyperref}
\hypersetup{
    colorlinks=true,     
    linkcolor=DenimBlue,    
    citecolor=FrenchPink,     
    filecolor=VividMulberry,      
    urlcolor=VividMulberry,       
}
\usepackage[symbol,flushmargin,hang]{footmisc}

\renewcommand{\footnoterule}{%
	\kern -3pt
	\hrule width \textwidth height 1pt
	\kern 2pt
}

\usepackage{amsmath}
\usepackage{amsthm}
\usepackage{amssymb}
\usepackage{mathrsfs}
\usepackage{mathtools}
\usepackage{etoolbox}
\usepackage{dsfont}
\usepackage{tikz-cd}

\usepackage{tabularx}
\usepackage{graphicx}
\usepackage{caption}
\usepackage{subcaption}
\usepackage{transparent}
\usepackage{float}
\usepackage{booktabs}
\usepackage{wrapfig}
\usepackage[final]{pdfpages}

\newtheoremstyle{thm}{}{}{\slshape}{}{\bfseries}{}{.5em}{}
\theoremstyle{thm}
\newtheorem{thm}{Theorem}
\newtheorem{prop}{Proposition}
\newtheorem{lem}{Lemma}
\newtheorem{cor}{Corollary}
\newtheorem*{conj*}{Conjecture}
\newtheorem*{defn*}{Definition}
\newtheorem*{thm*}{Theorem}

\newtheorem{thmout}{Theorem} 

\newtheoremstyle{ex}{}{}{}{}{\scshape}{{:}}{.5em}{}
\theoremstyle{ex}
\newtheorem*{ex*}{Example}
\newtheorem*{ack*}{Acknowledgements}

\newtheoremstyle{rem}{}{}{}{}{\scshape}{}{.5em}{}
\newtheorem{rem}{Remark}

\newtheoremstyle{pr}{}{}{}{}{\scshape}{{:}}{.5em}{}
\theoremstyle{pr}
\newtheorem*{pr*}{Proof}
\AtEndEnvironment{pr*}{\null\hfill\qedsymbol}

\renewenvironment{proof}[1][\proofname]{{\noindent\scshape #1:}}{\null\hfill\qedsymbol}

\usepackage{lipsum}

\usepackage{titlesec}
\renewcommand{\thesection}{\Roman{section}}
\titleformat{\section}[block]{\large\scshape}{\thesection.}{0.5em}{}
\renewcommand{\thesubsection}{\Roman{section}.\alph{subsection}}
\titleformat{\subsection}[block]{\scshape}{\thesubsection.}{0.5em}{}
\titleformat{\subsubsection}[block]{\bfseries}{}{0.5em}{}

\newcommand{\titre}{Hausdorff limits of submanifolds} 
\newcommand{\titrep}{of symplectic and contact manifolds} 
\newcommand{\titrepp}{Hausdorff limits of submanifolds of symplectic and contact manifolds}
\newcommand{\prenomauteur}{Jean-Philippe} 
\newcommand{\nomauteur}{Chass\'e} 

\newcommand{\pagetitre}[4]{
\noindent\rule{\linewidth}{1pt}
\begin{center}
\begin{tabular}{c}
{\LARGE\textmd{\scshape #1}}\\[2pt]
{\LARGE\textmd{\scshape #2}}\\[15pt]
{\large #3 {\scshape #4}$^*$}
\end{tabular}
\end{center}
\vspace{-0.2cm}
\noindent\rule{\linewidth}{1pt}}

\usepackage{fancyhdr}
\fancypagestyle{first}{
\pagestyle{fancy}
\fancyhead[C]{}
\fancyhead[RO,LE]{}
\fancyhead[LO,RE]{}
\fancyfoot[C]{}
\fancyfoot[L]{{\footnotesize Latest update: \today \\[1pt] $^*$This author's research was supported by NSERC and FRQNT doctoral research scholarships.}}

}

\makeatletter
\renewenvironment{abstract}{%
    \if@twocolumn
      \section*{\abstractname}%
    \else\small 
      \begin{center}%
        {\scshape \abstractname\vspace{-0.25cm}}
      \end{center}%
      \quotation
    \fi}
    {\if@twocolumn\else\endquotation\fi}
\makeatother

\makeatletter
\renewcommand\tableofcontents{%
	\begin{center}%
	{\scshape \contentsname\vspace{-0.25cm}}
	\end{center}%
	\@mkboth{\MakeUppercase\contentsname}{\MakeUppercase\contentsname}%
	\@starttoc{toc}%
}
\makeatother

\newcommand{\N}{\mathds{N}}
\newcommand{\Z}{\mathds{Z}}

\newcommand{\R}{\mathds{R}}

\newcommand{\D}{\mathds{D}}
\newcommand{\T}{\mathds{T}}
\newcommand{\Id}{\mathds{1}}
\newcommand{\Ker}{\operatorname{Ker}}

\newcommand{\del}{\partial}

\newcommand{\Vol}{\operatorname{Vol}}

\let\phi\varphi
\let\epsilon\varepsilon
\let\emptyset\varnothing

\begin{document}
\thispagestyle{first}

\pagetitre{\titre}{\titrep}{\prenomauteur}{\nomauteur}

\begin{abstract}
\noindent
We study sequences of immersions respecting bounds coming from Riemannian geometry and apply the ensuing results to the study of sequences of submanifolds of symplectic and contact manifolds. This allows us to study the subtle interaction between the Hausdorff metric and the Lagrangian Hofer and spectral metrics. In the process, we get proofs of metric versions of the nearby Lagrangian conjecture and of the Viterbo conjecture on the spectral norm. We also get $C^0$-rigidity results for a vast class of important submanifolds of symplectic and contact manifolds in the presence of Riemannian bounds. Likewise, we get a Lagrangian generalization of results of Hofer~\cite{Hofer1990} and Viterbo~\cite{Viterbo1992} on simultaneous $C^0$ and Hofer/spectral limits~---~even without any such bounds. \par
\end{abstract}
\vspace*{-.15cm}
\noindent\rule{\linewidth}{1pt}
\tableofcontents
\noindent\rule{\linewidth}{1pt}

\section{Introduction}
The main purpose of the present paper is to explore the relation between the classical Hausdorff metric on closed subsets and many of the metrics appearing in symplectic topology, e.g.\ the Lagrangian Hofer metric, the spectral metric and the recent shadow metrics of Biran, Cornea and Shelukhin \cite{BiranCorneaShelukhin2021}. More precisely, we want to explore the topology that these metrics induce on a given collection of Lagrangian submanifolds of a fixed symplectic manifold. In other words, we are interested in studying how convergence of a sequence in one metric affects the behavior of this sequence in the other ones. The reasoning behind the introduction of the Hausdorff metric into the list of considered metrics is twofold:
\begin{enumerate}[label=(\arabic*)]
	\item Contrary to the other metrics, its properties are well-known and easy to derive.
	\item It is defined on any choice of collection of (closed) Lagrangian submanifolds.
\end{enumerate}
The second point is of particular interest to us, as the shadow metrics are for example defined on Lagrangian submanifolds which might not even be of the same homotopy type. \par

This exploration was started in the author's previous work \cite{Chasse2021}. As noted in said work, there is however an obvious problem with introducing the Hausdorff metric: in full generality, there is no relation between the Hausdorff metric and the metrics coming from symplectic topology. Nonetheless, when one only consider sequences respecting certain bounds coming from an auxiliary Riemannian metric, the behavior on each side become intimately related. This is because such bounds essentially stop sequences from Hausdorff-converging to pathological spaces. On the other hand, as we shall see below, our Riemannian bounds implies uniform bounds on sectional curvature and injectivity radius. In view of Cheeger's finiteness theorem~\cite{Cheeger1970}, this thus implies that elements in the sequence can only be of finitely many diffeomorphism type. However, since we are ultimately interested in sequences having \emph{some} Riemannian bound rather than a precise one, it is unclear how this affects the possible limit of such a sequence. \par

In our previous work, we studied sequences of Lagrangian submanifolds converging in some nice metrics coming from symplectic topology, and we proved that they must also converge in the Hausdorff metric. We now turn to the opposite problem: if we have a Hausdorff-converging sequence of Lagrangian submanifolds, what can we say of its behavior in those nice metrics coming from symplectic topology. \par

\begin{thmout} \label{thm-out:section}
	If $\{L_i\}$ is a Riemannianly-bounded sequence of exact Lagrangian submanifolds in $T^*L$ which Hausdorff-converges to the image of the zero section, then $L_i$ is the graph of a 1-form for $i$ large enough. Without the exactness assumption, whether this is true or not depends solely on the first Betti number of the $L_i$'s and of $L$.
\end{thmout}

The idea that such a statement should hold was first shared with us by Lalonde during a discussion regarding previous work. Note that we prove the statement for a slightly weaker hypothesis than convergence in the Hausdorff metric. This gives as a corollary metric versions of the nearby Lagrangian conjecture and of the Viterbo conjecture on the spectral metric, as we shall see below. \par

Furthermore, through a direct computation, this implies that $\{L_i\}$ also converges to $L$ in the Lagrangian Hofer metric. Therefore, in the exact case, convergence in the Hausdorff metric is the same thing as convergence in a nice metric coming from symplectic topology when Riemannian bounds are present. Indeed, all known such metrics are bounded from above by the Lagrangian Hofer metric, and thus convergence in the later metric implies convergence in the other ones. This thus gives a complete characterization of a small-enough neighborhood of an exact Lagrangian submanifold in any of these nice metrics when Riemannian bounds are present: they are just graph deformations of the submanifold. \par

As we shall see below, the exactness condition on $\{L_i\}$ is necessary, as $L_i$ could be a nontrivial covering over $L$ without this condition. Note however that even without the exactness condition, there are still rigidity phenomena which are not present for non-Lagrangian submanifolds. Therefore, the present result is truly in the realm of symplectic topology. \par

We also analyze the type of possible limits that a Hausdorff-converging sequence of Lagrangian submanifolds might have. In particular, this allows us to even better understand limits in those nice metrics coming from symplectic topology when Riemannian bounds are imposed. We show that when Riemannian bounds are present, the limit must be the image of a Lagrangian immersion, although some regularity might be lost in the process. In fact, the techniques that we use there apply to not just Lagrangian submanifolds, but also a large class of important submanifolds of symplectic and contact manifolds. \par

\begin{thmout} \label{thm-out:eliashberg-gromov}
	Hausdorff limits of sequences of (co)isotropic submanifolds of a given symplectic or contact manifold~---~respecting appropriate Riemannian bounds~---~are $C^1$-immersed (co)isotropic submanifolds. Furthermore, in the Lagrangian case, exactness, weak exactness and monotonicity are preserved in the limit when said limit is smoothly embedded.
\end{thmout}

We will also explore the possible Hausdorff limits with lighter Riemannian bounds, and even no Riemannian bounds at all. The later exploration is of particular interest to $C^0$-symplectic topology, as it proposes a new possible definition of what a $C^0$-Lagrangian submanifold should be. \par

\subsection{Precise statements}
Throughout the paper, we fix a complete Riemannian manifold $(M,g)$. We will consider immersions $f:N\looparrowright M$, where $N$ is closed and connected. We denote by $B_f$ its associated second fundamental form\index{seconde forme fondamentale} and by $\Vol(f)$\index{volume} the volume of $N$ with respect to the metric $f^*g$. For $k\in\N$, $\Lambda\in [0,\infty)$, and $V\in (0,\infty)$, we consider $\mathscr{I}_k(\Lambda,V)$, the space of such immersions $f:N\looparrowright M$, where $\dim N=k$, 
\begin{align*}
||B_f||\leq\Lambda, \qquad\text{and}\qquad \Vol(f)\leq V.
\end{align*}

In what follows, we will often take $M=T^*L$, where $L$ is a closed connected Riemannian $n$-manifold, and $T^*L$ is equipped with the associated Sasaki metric. Using this notation, we can give a precise statement for Theorem \ref{thm-out:section}. \par

\begin{thm} \label{thm:section}
	Let $\Lambda\geq 0$ and $V>0$. Let $\{f_i:L_i\hookrightarrow T^*L\}\subseteq\mathscr{I}_n(\Lambda,V)$ be a sequence of exact Lagrangian embeddings. Suppose that $f_i(L_i)$ sits in the codisk bundle $D^*_{r_i}L$ of radius $r_i$ and that $\{r_i\}$ tends to $0$. Then, $f_i(L_i)$ is the graph of a 1-form for $i$ large enough. \\
	Furthermore, if the $f_i$'s are not (necessarily) exact, then the above conclusion holds if and only if $L_i$ and $L$ have the same first Betti number for $i$ large.
\end{thm}

The main tool in the proof of this result is a theorem of Shen \cite{Shen1995} proving some sort of precompacity result for $\mathscr{I}_k(\Lambda,V)$. In fact, Shen's result proves that $\mathscr{I}_k(\Lambda,V)$ can be naturally compactified using $C^{1,\alpha}$-immersions, for any $\alpha\in (0,1)$. Together with the fact that the projection $f_i(L_i)\to L$ must be a homotopy equivalence \cite{AbouzaidKragh2018}, this gives Theorem \ref{thm:section}. \par

The main application of this result is in proving metric versions of the nearby Lagrangian conjecture and of the Viterbo conjecture on the spectral norm. \par

\begin{cor} \label{cor:conjectures}
	Let $L$ be a closed connected Riemannian $n$-manifold. There exist constants $A=A(L)>0$ and $R=R(L,\Lambda,V)>0$ with the following property. Let $L'$ be an exact Lagrangian submanifold of the codisk bundle 
	\begin{align*}
	D^*_R L:=\{(x,v)\in T^*L\ |\ |v|\leq R\}
	\end{align*}
	such that the inclusion $L'\hookrightarrow D^*_R L\hookrightarrow T^*L$ is in $\mathscr{I}_n(\Lambda,V)$. Then, 
	\begin{enumerate}[label=(\roman*)]
		\item $L'$ is Hamiltonian isotopic to the zero section; \\
		\item the spectral norm respects the inequality
		\begin{align*}
		\gamma(L,L')\leq A.
		\end{align*}
	\end{enumerate}
\end{cor}

\begin{rem} \label{rem:conjectures_history}
	We briefly review the advances made in proving both conjectures in full generality, i.e.\ without any Riemannian bounds.
	\begin{enumerate}[label=(\arabic*)]
		\item The nearby Lagrangian conjecture, i.e.\ (i) in Corollary \ref{cor:conjectures}, is only known when $L$ is $S^1$ (folklore), $S^2$ (follows from work of Hind \cite{Hind2004}), $\R P^2$ (follows from work of Hind, Pinsonnault, and Wu \cite{HindPinsonnaultWu2013}), and $\T^2$ (proved by Dimitroglou Rizell, Goodman and Ivrii \cite{DimitroglouGoodmanIvrii2016}). However, great advancement has been made in proving the conjecture in full generality. As noted before, it is known that the canonical projection $\pi:T^*L\to L$ must induce a simple homotopy equivalence $L'\to L$ \cite{AbouzaidKragh2018}. Likewise, $L'$ and $L$ must be isomorphic objects in the Fukaya category of $T^*L$ when $L$ is spin \cite{FukayaSeidelSmith2008i, FukayaSeidelSmith2008ii, Nadler2009}. \par
		
		\item The Viterbo conjecture on the spectral norm, i.e.\ (ii) in Corollary \ref{cor:conjectures}, has recently been proven for a large class of nice manifolds by Shelukhin \cite{Shelukhin2018, Shelukhin2019}, and for another large class of manifolds by Viterbo~ \cite{Viterbo2022}, and Guillermou and Vichery~\cite{GuillermouVichery2022}, indepently. \par
	\end{enumerate}
\end{rem}

From the statement of Theorem~\ref{thm:section} on nonexact Lagrangian submanifolds, we also get the following surprising dichotomy for sequences of embedded Lagrangian submanifolds. \par

\begin{cor} \label{cor:sequence_betti}
	Let $M$ be a symplectic manifold, and let $\{f_i:L_i\hookrightarrow M\}\subseteq\mathscr{I}_n(\Lambda,V)$ be a sequence of Lagrangian embeddings such that $\{f_i(L_i)\}$ Hausdorff-converges to a Lagrangian submanifold $L$. Then, passing to a subsequence, exactly one of the following is true:
	\begin{enumerate}[label=(\alph*)]
		\item $f_i(L_i)$ is Lagrangian isotopic to $L$;
		\item $b_1(L_i)>b_1(L)$,
	\end{enumerate}
	where $b_1$ denotes the first Betti number.
\end{cor}

In fact, as we shall see below in Theorem~\ref{thm:eliashberg-gromov}, the hypothesis that $L$ be Lagrangian may be dropped, as it is automatically satisfied. \par

Another consequence of Theorem \ref{thm:section} is in completing the author's previous work on the equivalence of the topologies induced by various metrics on an appropriate space of Lagrangian submanifolds. \par

\begin{cor} \label{cor:hausdorff_to_hofer}
	Let $\lambda$ be a Liouville form of an\ exact symplectic manifold $M$, and let $\{L_i\subseteq M\}$ be a sequence of $\lambda$-exact Lagrangian submanifolds such that the inclusions are in some fixed $\mathscr{I}_n(\Lambda,V)$. If $\{L_i\}$ Hausdorff-converges in $M$ to a (smooth) $\lambda$-exact Lagrangian submanifold $L_0$, then it also converges to $L_0$ in the Lagrangian Hofer metric $d_H$. In particular, it also converges to $L_0$ in the spectral norm and in any shadow metric.
\end{cor}

In fact, as we shall see below, the exactness requirement on the limit is superfluous, as the Hausdorff limit of a sequence of $\lambda$-exact Lagrangian immersions is automatically itself exact for the same Liouville form. \par

In fact, Shen's result allows us to explore sequences of other important type of submanifolds in symplectic and contact topology, not just Lagrangian ones. Indeed, since the result gives some sort of compacity in $C^{1,\alpha}$-topology, $0<\alpha<1$, many of these properties are preserved in the limit. This leads to rigidity results similar in flavor to the celebrated theorem of Gromov and Eliashberg on the $C^0$-closedness of the group of symplectomorphisms in the group of diffeomorphisms \cite{Gromov1986, Eliashberg1987}, but with the additional requirement of there being Riemannian bounds. \par

\begin{thm} \label{thm:eliashberg-gromov}
	Let $\{f_i:L_i\looparrowright M\}\subseteq \mathscr{I}_k(\Lambda,V)$ be a sequence of Lagrangian (resp.\ isotropic, coisotropic, Legendrian, contact isotropic, or contact coisotropic) immersions. Suppose that $\{f_i(L_i)\}$ Hausdorff-converges to a closed subset $N$. Then, $N$ is the image of a Lagrangian (resp. isotropic, coisotropic, Legendrian, contact isotropic, or contact coisotropic) $C^{1,\alpha}$-immersion $f_0:L_0\looparrowright M$, where $L_0$ is closed and connected (and of dimension $k$). \par
	
	Furthermore, if the $f_i$'s are exact for some Liouville form $\lambda$, then there is a $C^{1,\alpha}$-function $h_0:L_0\to\R$ such that $f^*\lambda=dh_0$. \par 
	
	Likewise, if the $f_i$'s are weakly exact (resp.\ monotone with monotonicity constant $\rho>0$), and $f_0$ is an embedding, then $f_0$ is weakly exact (resp.\ monotone with monotonicity constant $\rho$).
\end{thm}

Apart from the statement on weakly exact and monotone Lagrangian submanifolds, each element of Theorem~\ref{thm-out:eliashberg-gromov} has an equivalent statement for noncompact submanifolds~---~of possibly infinite volume. The proof relies a pointed version of Shen's result and a small technical trick. For ease of presentation, we do not give here the  the precise details of the equivalent statements and delay their presentation to Subsection~\ref{subsec:rigidity_without_volume_bounds}.

\begin{rem} \label{rem:gen_eliash-gromov}
	This theorem is related to previously-known $C^0$-rigidity results.
	\begin{enumerate}[label=(\arabic*)]
		\item This result can be viewed as having some similarity to Laudenbach and Sikorav's result \cite{LaudenbachSikorav1994} on the $C^0$-rigidity of Lagrangian embeddings (under some technical assumptions). This result was recently upgraded to general Lagrangian submanifolds and to a class of nice Legendrian submanifolds by Nakamura \cite{Nakamura2020}. The great improvement here is that we allow embeddings~---~and in fact, even immersions~---~of varying domains; the price to pay are bounds coming from Riemannian geometry. We will however see below that we can partially get rid on the volume bound, and we will show some $C^0$-rigidity result without any type of Riemannian bounds.
		
		\item Likewise, note that when $f_i$ is the inclusion of the graph of some symplectomorphism $\psi_i:M\xrightarrow{\sim} M$, i.e.\ when $f_i(M)$ is a Lagrangian graph, then Hausdorff-convergence of $\{f_i(M)=\operatorname{graph}\psi_i\}$ to $N=\operatorname{graph}\psi$ is equivalent to $C^0$-convergence of $\{\psi_i\}$ to $\psi$. Furthermore, uniform $C^2$-bounds on $\{\psi_i\}$ implies the existence of a $C^1$-converging subsequence by the Arzela-Ascoli theorem, and thus the limit $\psi$ is a $C^1$-symplectomorphism. However, such bounds also implies that $\{f_i\}$ is in $\mathscr{I}_n(\Lambda,V)$ for some $\Lambda\geq 0$ and $V>0$. Therefore, Theorem \ref{thm:eliashberg-gromov} can also be seen as a generalization of that simple fact.
	\end{enumerate}
\end{rem}

We end this introduction by showcasing the previously--mentioned rigidity result without Riemannian bounds generalizing results of Hofer~\cite{Hofer1990} and Viterbo~\cite{Viterbo1992} on simultaneous $C^0$ and Hofer/spectral limits.

\begin{thm} \label{thm:equiv_limits}
	Let $\{L_i\}$ be a sequence of closed connected Lagrangian submanifolds in a $2n$-dimensional symplectic manifold $M$. Let $\hat{d}^{\mathscr{F},\mathscr{F'}}$ be a Chekanov-type metric on some collection of Lagrangian submanifolds $\mathscr{L}^\star(M)$ (c.f.\ \cite{Chasse2021}), and suppose that $L_i\in\mathscr{L}^\star(M)$ for all $i$. Furthermore, suppose the following:
	\begin{enumerate}[label=(\arabic*)]
		\item The sequence Hausdorff-converges to a closed topological submanifold $N$ of dimension at most $n$.
		\item The sequence converges to some $L_0\in \mathscr{L}^\star(M)$ in $\hat{d}^{\mathscr{F},\mathscr{F'}}$.
	\end{enumerate}
	Then, $N=L_0$.
\end{thm}

Concretely, one can take $\hat{d}^{\mathscr{F},\mathscr{F'}}$ to be the Lagrangian Hofer metric, the spectral metric, or one of the shadow metrics, with $\mathscr{L}^\star(M)$ being then the set of all weakly exact (or monotone with good conditions; see again \cite{Chasse2021}) Lagrangian submanifolds.

\subsection{Organization of the paper}
The rest of the paper is divided in two parts. The first one is mainly concerned with proving Theorem~\ref{thm:section} and exploring the rigidity phenomenon underlying it. More precisely, we begin by studying sequences of not-necessarily-Lagrangian immersions which behave well in the codomain (Subsection~\ref{subsec:sequences_nice-image}), then we prove the statement of Theorem~\ref{thm:section} in the exact setting (Subsection~\ref{subsec:app_lag}), and we end by showing that this is truly a Lagrangian phenomenon (Subsection \ref{subsec:rigidity_lag}). It is also in Subsection~\ref{subsec:rigidity_lag} that we prove the statement of Theorem~\ref{thm:section} in the nonexact setting, along with Corollary~\ref{cor:sequence_betti}. The second part is not only concerned with proving Theorem~\ref{thm:eliashberg-gromov}, but also with relating it to the author's previous work (Subsection~\ref{subsec:tame+volume}), extending parts of it to the case $V=\infty$ (Subsection \ref{subsec:rigidity_without_volume_bounds}), and exploring rigidity phenomena beyond Riemannian bounds (Subsection~\ref{subsec:rigidity_without_bounds}) by proving Theorem~\ref{thm:equiv_limits}. \par

\subsection{Acknowledgments}
This research was part of my PhD thesis and was financed by a NSERC and a FRQNT scholarship. I would like to thank my advisor, Octav Cornea, for his continued interest in my research and for the many insightful discussions that we have had. I would also like to thank Sobhan Seyfaddini for his observations regarding $C^0$-converging sequences of Hamiltonian diffeomorphisms. Finally, I am indebted to Dominique Rathel-Fournier for sharing his insight on coverings and pointing out to me Polterovich's construction of nontrivial Lagrangian coverings in the cotangent bundle of flat manifold.

\section{Sequences of immersions} \label{sec:sequences}
The focus of this section is the proof of Theorem \ref{thm:section} and the study of the rigidity of Lagrangian embeddings in general. We thus begin with a general study of sequences of immersions, then we apply this new knowledge to Lagrangian embeddings specifically, and finally explore how this is truly a Lagrangian phenomenon. \par

As mentioned above, we will make great use of Shen's result on sequences of immersions. Therefore, we thought that it could be useful for the reader to simply write the explicit statement here before moving on. \par

\begin{thm*}[\cite{Shen1995}]
	Let $\{f_i:N_i\looparrowright M\}_{i\geq 1}\subseteq \mathscr{I}_k(\Lambda,V)$ be such that $\cup_i f_i(N_i)$ is contained in a compact subset of $M$. Let $0<\alpha'<\alpha<1$. Then, we have the following:
	\begin{enumerate}[label=(\roman*)]
		\item a subsequence, still denoted $\{f_i\}$;
		\item a closed connected smooth manifold $N_0$;
		\item a Riemannian metric $g_0$ on $N_0$ of class $C^{1,\alpha}$;
		\item an immersion $f_0:N_0\looparrowright M$ of class $C^{1,\alpha}$ with $f_0^*g=g_0$;
		\item for each $i$ large enough, a diffeomorphism $\phi_i:N_0\xrightarrow{\sim} N_i$ of class $C^{2,\alpha}$
	\end{enumerate}
	such that $\{f_i\circ\phi_i\}$ converges in the $C^{1,\alpha'}$ topology to $f_0$.
\end{thm*}

Essentially, this theorem is the appropriate generalization to (immersed) submanifolds of the Gromov--Hausdorff compactness theorem \cite{Gromov1981, Katsuda1985} for Riemannian manifolds with uniformly bounded sectional curvature and injectivity radius~---~the latter bound comes from the bound on the second fundamental form in the submanifold case. \par

\subsection{Immersions with converging images} \label{subsec:sequences_nice-image}
In this subsection, we use Shen's theorem to study sequences $\{f_i\}$ of immersions in $\mathscr{I}_k(\Lambda,V)$ which have images behaving well with respect to the Hausdorff metric of $M$. \par

For a closed connected (embedded) submanifold $N$ of $M$ of dimension $k=\dim N_i$ and $r>0$, denote by $B_r(N)$ its tubular neighborhood of size $r$. In this section, we will suppose that there is a sequence $\{r_i\}\subseteq\R_{>0}$ converging to 0 such that
\begin{align} \label{eqn:condition}
f_i(N_i)\subseteq B_{r_i}(N) \qquad\forall i \tag{$\dagger$}.
\end{align}
In other words, if we define for subsets $A,B\subseteq M$
\begin{align*}
s(A;B):=\sup_{x\in A} d_M(x,B):=\sup_{x\in A}\inf_{y\in B} d_M(x,y),
\end{align*}
then Property (\ref{eqn:condition}) is equivalent to $\lim_{i\to\infty} s(f_i(N_i);N)=0$. Remember that the Hausdorff metric of $M$ is given by $\delta_H(A,B)=\max\{s(A;B),s(B;A)\}$. Thus, this condition is \textit{a priori} strictly weaker than Hausdorff convergence to $N$. \par

\begin{lem} \label{lem:hausdorff}
	If Property (\ref{eqn:condition}) holds, then $\{f_i(N_i)\}$ converges to $N$ in the Hausdorff metric.
\end{lem}
\begin{pr*}
	Property (\ref{eqn:condition}) implies that there exists a Hausdorff-converging subsequence $\{f_i(N_i)\}$ and that such a sequence must have as limit a subset $E\subseteq N$. On the other hand, there exists by Shen's theorem yet another subsequence $\{f_i\}$ which $C^{1,\alpha'}$-converge to a $C^{1,\alpha}$-immersion $f_0:N_0\looparrowright M$. We must then have $f_0(N_0)=N$. Indeed, we may see $f_0$ as an immersion into $N$. But by the inverse function theorem, this immersion is open. Since $N_0$ is compact, $f_0$ is also closed. Therefore, $E=f_0(N_0)$ is clopen; it must thus be equal to $N$. \par
	
	Suppose now that $\{f_i(N_i)\}$ did not converge to $N$. Then we would have a subsequence $\{f_i(N_i)\}$ such that $s(N;f_i(N_i))\geq\epsilon$ for some $\epsilon>0$. We then get a contradiction by passing to a converging subsequence $\{f_i(N_i)\}$, since we have just proven that its Hausdorff limit must be $N$.
\end{pr*}

Lemma \ref{lem:hausdorff} allows us to prove the main technical result of this subsection. \par

\begin{prop} \label{prop:reparam}
	Let $\{f_i:N_i\hookrightarrow M\}\subseteq\mathscr{I}_k(\Lambda,V)$ be such that Property (\ref{eqn:condition}) holds. Denote by $\iota:N\hookrightarrow M$ the inclusion of the limit manifold. For all $\epsilon>0$, there exists $I\in\N$ such that if $i\geq I$, then there exists a finite $C^1$-covering\footnote[2]{By $C^1$-covering, we here mean that both $p_i$ and its local trivalizations $p_i^{-1}(U)\to U\times G$ are of class $C^1$. The finiteness means that $|G|<\infty$.} $p_i:N_i\to N$ such that
	\begin{align*}
	d_{C^{1,\alpha'}}(f_i,\iota\circ p_i)<\epsilon.
	\end{align*}
\end{prop}
\begin{pr*}
	Take a $C^{1,\alpha'}$-converging subsequence $\{f_i\}$~---~which exists by Shen's result~---~and denote by $f_0:N_0\looparrowright M$ its limit. By Lemma \ref{lem:hausdorff}, the image of $f_0$ is $N$. We thus get a map $p:N_0\to N$ making the diagram
	\begin{center}
		\begin{tikzcd}[row sep=6pt]
			N_0 \arrow[rrd,"f_0"] \arrow[dd,dotted,"p"'] & & \\
			& & M \\
			N \arrow[rru,hook,"\iota"] & &
		\end{tikzcd}
	\end{center}
	commute by inverting $\iota$ on the image of $f_0$. \par
	
	Note that $p$ is a surjective submersion between closed manifolds. In particular, $p$ is necessarily proper. Therefore, by Ehresmann's fibration theorem, $p$ must be a locally trivial fibration over $N$. For dimensional reasons, $p$ must thus be a covering. We can then take $p_i=p\circ\phi_i^{-1}$, where $\phi_i:N_0\to N_i$ is a diffeomorphism. \par
	
	Note that even though Ehresmann's theorem is usually stated for smooth maps, there is a version of the theorem for $C^1$ maps due to Earle and Eells~\cite{EarleEells1967} in the much more general setting of Finsler manifolds modelled on a Hilbert space (see their section 4(A)). This thus allow us to use the theorem even though $p$ is only \textit{a priori} of class $C^{1,\alpha}$. However, if $f_0$ is smooth, then so are $p$ and the $p_i$'s. \par
	
	This thus implies the result for any converging subsequence. Suppose that the statement is not true for the sequence $\{f_i\}$ itself. Then, we get a subsequence $\{f_i\}$ such that $d_{C^{1,\alpha'}}(f_i,\iota\circ p)\geq\epsilon$ for some $\epsilon>0$ and for all coverings $N_i\to N$. Passing to a converging subsequence, we clearly get a contradiction.
\end{pr*}

\begin{rem} \label{rem:covering}
	Of course, $p$~---~and thus $p_i$ for $i$ large~---~must be a diffeomorphism whenever $N$ is simply connected. However, in full generality, it was pointed to us by Dominique Rathel-Fournier that it is entirely possible for $p$ to be a covering~---~even in the nicest of cases. For example, one can consider the sequence of embeddings
	\begin{center}
		\begin{tikzcd}[row sep=0pt,column sep=1pc]
			f_i\colon \T^2 \arrow{r} & \T^2\times\R^2=T^*\T^2 \\
			{\hphantom{\alpha\colon{}}} (\theta_1,\theta_2) \arrow[mapsto]{r} & {\left(2\theta_1,\theta_2,\frac{1}{i}\cos\theta_1,\frac{1}{i}\sin\theta_1\right)}.
		\end{tikzcd}
	\end{center}
	Clearly, this is a sequence having Property (\ref{eqn:condition}) for $N=\T^2\times\{0\}$. Furthermore, a direct computation gives that this sequence is in $\mathscr{I}_2(1/\sqrt{5},4\sqrt{5}\pi^2)$. However, the associated map $p:N_0\to N$ is the double cover $(\theta_1,\theta_2)\mapsto (2\theta_1,\theta_2)$.
\end{rem}

In some sense, $p$ measures the difference between the abstract Gromov--Hausdorff (or equivalently Lipschitz, see Theorem~8.25 of \cite{Gromov1981}) limit of the sequence of geometrically bounded Riemannian manifolds $\{(N_i,f_i^*g)\}$ and the classical Hausdorff limit of the sequence of compact subsets $\{f_i(N_i)\}$ in $(M,g)$. Indeed, $N_0$ is precisely the first limit (e.g.\ $N_0$ is isometric to $\R^2/(2\Z\times\Z)$ in the example of Remark~\ref{rem:covering}), whilst $N$ is the second one (e.g.\ $N$ is $\R^2/\Z^2$ in Remark~\ref{rem:covering}). Proposition~\ref{prop:reparam} tells us precisely that these two limits are related by a finite covering. The source of this difference in limits is the fact that the map $f_i:(N_i,d_{N_i})\to (f(N_i),d_M)$ is not in general a metric isometry. Here, $d_M$ is the restriction of the metric on $M$ to $f(N_i)$, and $d_{N_i}$ is the metric on $N$ induced by the Riemannian metric $f_i^*g$. Indeed, we only know that $f_i$ is 1-Lipschitz. \par

\subsection{Proof of the first part of Theorem \ref{thm:section}} \label{subsec:app_lag}
We now apply the above results to Lagrangian submanifolds and prove metric versions of the nearby Lagrangian conjecture and of the Viterbo conjecture on the spectral norm. Therefore, from now on, we suppose that $M=T^*L$ for some $n$-dimensional closed connected Riemannian manifold $L$. We equip $T^*L$ with the standard symplectic form, almost complex structure and metric. \par

We need to prove the Theorem \ref{thm:section}. Note that it gives us a proof of Corollary \ref{cor:conjectures} right away; the proof of Corollary \ref{cor:hausdorff_to_hofer} will only be given in Subsection \ref{subsec:compactness}. \par

\begin{proof}[Proof of Corollary \ref{cor:conjectures}]
	Let $\{f_i\}$ be as in Theorem \ref{thm:section}. Since $f_i(L_i)$ is an exact Lagrangian graph, it must be the graph of an exact 1-form $dh_i$. We take the vector field $X_i$ defined via $\iota_{X_i}\omega=\beta\pi^*dh_i$, where $\beta$ is a compactly supported bump function which is identically 1 on $B_{r_1}(L)=D_{r_1}^* L$. Here, $\pi:T^*L\to L$ denotes the canonical projection. Then, $X_i$ generates a compactly supported Hamiltonian isotopy sending the zero section to $f_i(L_i)$. \par
	
	The fact that $f_i(L_i)$ is a graph also implies that the Floer complex $CF(f_i(L_i),T^*_x L)$ has only one generator for any $x\in L$. In particular, its boundary depth is zero. Therefore, by work of Biran and Cornea \cite{BiranCornea2021}, we also get $\gamma(L,f_i(L_i))\leq A$. Here, $A$ is the constant appearing in the work of Biran and Cornea associated to $D_1^* L$, which $f_i(N_i)$ is in for $i$ large enough. \par
	
	The proof concludes by contradiction: if a $R\leq 1$ as in the theorem did not exist, we would then have a sequence of exact Lagrangian embeddings respecting Property (\ref{eqn:condition}), but not respecting the conclusions of the theorem. This would be a contradiction with the above paragraph.
\end{proof}

We now give a proof of statement in Theorem~\ref{thm:section} on sequences of exact Lagrangian submanifolds. \par

\begin{proof}[Proof of the first part of Theorem~\ref{thm:section}]
	By work of Abouzaid and Kragh \cite{AbouzaidKragh2018}, the composition $L_i\hookrightarrow T^*L \to L$ is a (simple) homotopy equivalence. In particular, it is an isomorphism on the fundamental group. \par
	
	On the other hand, by Proposition \ref{prop:reparam}, $f_i$ must be transverse to every fiber for $i$ large enough. Therefore, $\pi|_{f_i(L_i)}$ must be a covering onto $L$. However, by the above paragraph, the covering is trivial, i.e.\ $\pi|_{f_i(L_i)}$ is a diffeomorphism. Therefore, $f_i(L_i)$ must be the graph of a 1-form for $i$ large enough.
\end{proof}

\begin{rem} \label{rem:section-vs-covering}
	\hspace{2em}
	\begin{enumerate}[label=(\arabic*)]
		\item The proof of Theorem \ref{thm:section} applies for any simply connected complete Riemannian manifold $N$. We then get that $f_i(N_i)$ is the graph of a section of the normal bundle of $N$ in $M$. More generally, without any topological assumption on $N$, $f_i(N_i)$ admits a lift $\widetilde{N}_i$ in the normal bundle of $N$ in $f_0^*TM$, and this lift is the graph of a section of that normal bundle. Note that this is true even when neither $f_i$ nor the limit $f_0$ is an embedding.
		\item As we have seen in Remark \ref{rem:covering} however, Theorem \ref{thm:section} is not true for non-Lagrangian embeddings. In fact, Theorem \ref{thm:section} is typically not even true for \emph{nonexact} Lagrangian submanifolds. Indeed, Polterovich \cite{Polterovich1990} constructed for any closed flat manifold $W\neq\T^n$ Lagrangian tori in $T^*W$ having the property that the composition $\T^n\hookrightarrow T^*W\to W$ is a nontrivial cover. These tori can be realized as the image under the natural map $T^*\T^n\to T^*W$ of the graph of any constant 1-form on $\T^n$. Therefore, when we equip $\T^n$ and $W$ with the flat metric and their cotangent bundle with the corresponding Sasaki metric, the tori are totally geodesic, have the same volume as $\T^n$, and can be taken to be arbitrary close to the zero section of $T^*W$. In other words, we get a sequence of Lagrangian embeddings in $\mathscr{I}_n(0,\mathrm{Vol}(\T^n))$ converging in $T^*W$ to a nontrivial covering of $W$. In particular, the second possibility in Corollary~\ref{cor:sequence_betti} can indeed happen. We will explore this kind of phenomenon in more details in the next subsection.
	\end{enumerate}
\end{rem}

\subsection{Rigidity of Lagrangian embeddings} \label{subsec:rigidity_lag}
As we are studying Riemannian and symplectic phenomena at the same time, it can be hard to parse what comes from the Riemannian bounds and what comes from the Lagrangian condition. In fact, it could \textit{a priori} be the case that a result analogous to Theorem \ref{thm:section} exists for an appropriate class of non-Lagrangian submanifolds. Indeed, as noted in Remark \ref{rem:section-vs-covering}, it seems that the exactness condition~---~not just the Lagrangian condition~---~is required for the result. In this subsection, we thus want to dispel the idea that this could be an entirely non-Lagrangian phenomenon. \par

We begin by exploring some basic properties of Lagrangian embeddings. \par

\begin{prop} \label{prop:lagrangian_coverings}
	Let $f:L'\hookrightarrow T^*L$ be a Lagrangian embedding such that $\pi\circ f:L'\to L$ is a finite covering. Then, $\pi\circ f$ is a diffeomorphism if and only if the first Betti numbers of $L$ and $L'$ are the same, i.e.\ $b_1(L')=b_1(L)$.
\end{prop}
\begin{pr*}
	One implication is of course trivial. Suppose therefore that $b_1(L')=b_1(L)$. Since $\pi\circ f$ is a finite covering, we know that $(\pi\circ f)^*:H^1(L;\R)\to H^1(L';\R)$ is injective. By the condition on the Betti numbers, it is thus an isomorphism. \par
	
	Let $\sigma':=f^*\lambda$, where $\lambda$ is the canonical 1-form on $T^*L$. Note that $\sigma'$ is closed since $f$ is Lagrangian. By the above paragraph, there exists a 1-form $\sigma$ on $L$ such that $(\pi\circ f)^*[\sigma]=[\sigma']$. Let $\{\psi_t\}$ be the symplectic isotopy generated by the vector field $X$ defined by $\iota_X\omega=-\pi^*\sigma$. We then have
	\begin{align*}
	\left[(\psi_1\circ f)^*\lambda\right] = \left[f^*\lambda + \int_0^1\iota_X\omega\right] = [\sigma']-f^*\pi^*[\sigma]=0,
	\end{align*}
	where we have made use of Cartan's magic formula and the definition of the Lie derivative for vector fields. In other words, it must be that $\psi_1\circ f:L'\hookrightarrow T^*L$ is an exact Lagrangian embedding. As previously noted, $\pi\circ\psi_1\circ f$ must then be an isomorphism on the fundamental group. Therefore, the same holds for $\pi\circ f=\pi\circ\psi_0\circ f$; it must thus be a diffeomorphism.
\end{pr*}

Combining Propositions \ref{prop:reparam} and \ref{prop:lagrangian_coverings}, we directly get the following result.

\begin{cor} \label{cor:nearby_symp}
	Let $L$ be a closed connected Riemannian manifold such that any finite covering $L'\to L$ is such that $b_1(L')=b_1(L)$, e.g.\ $\pi_1(L)$ is free, abelian free, or finite. Let $\Lambda\geq 0$ and $V>0$. There exists $R>0$ such that whenever $f\in\mathscr{I}_n(\Lambda,V)$ is a Lagrangian embedding with image in $D^*_R L$, then said image is symplectomorphic to the zero section.
\end{cor}

Likewise, Proposition~\ref{prop:lagrangian_coverings} allows us to prove the nonexact statement of Theorem~\ref{thm:section} and Corollary~\ref{cor:sequence_betti}.

\begin{proof}[Proof of the second part of Theorem~\ref{thm:section}]
	By Proposition~\ref{prop:reparam}, $\pi\circ f_i$ is a finite covering onto $L$ for $i$ large. By Proposition~\ref{prop:lagrangian_coverings}, this covering is a diffeomorphism, i.e.\ $f_i(L_i)$ is the graph of a 1-form, if and only if $b_1(L_i)=b_1(L)$, which proves the statement.
\end{proof}

\begin{proof}[Proof of the second part of Theorem~\ref{thm:section}]
	Again, by Proposition~\ref{prop:reparam}, $\pi\circ f_i$ is a finite covering onto $L$ for $i$ large. As noted in the proof of Proposition~\ref{prop:lagrangian_coverings}, $(\pi\circ f)^*:H^1(L;\R)\to H^1(L';\R)$ is then injective. In particular, $b_1(L_i)\geq b_1(L)$. We can thus find a subsequence such that either $b_1(L_i)=b_1(L)$ or $b_1(L_i)> b_1(L)$. By Theorem~\ref{thm:section}, the first case is equivalent to $f_i(L_i)$ being the graph of a 1-form $\sigma_i$ in a Weinstein neighborhood $\Psi$ of $L$. Then, $\{L_t:=\Psi(\mathrm{graph}\ t\sigma_i)\}$ gives the required Lagrangian isotopy.
\end{proof}

\begin{rem} \label{rem:symp_vs_lag_iso}
	Even though there exists a symplectic isotopy in $T^*L$ sending $L$ to $\mathrm{graph}\ \sigma_i$, e.g.\ the one generated by $\pi^*\sigma_i$, this isotopy is not compactly supported. In fact, there cannot be compactly-supported symplectic isotopy sending $L$ to $\mathrm{graph}\ \sigma_i$ if $\sigma_i$ is not exact: since $H^1_c(T^*L)=0$ (when $L\neq S^1$), the flux short exact sequence implies that $\mathrm{Ham}_c(T^*L)=\mathrm{Symp}_{0,c}(T^*L)$. However, there is no guaranty that this noncompactly-supported symplectic isotopy in a Weinstein neighborhood can be extended to the whole ambient symplectic manifold. The obvious exception to this is when the ambient manifold is itself a cotangent bundle and $L$ is (in the Hamiltonian orbit of) the 0-section.
\end{rem}

Note that the equivalent results to Corollaries~\ref{cor:sequence_betti} and~\ref{cor:nearby_symp} in the smooth category are entirely false, as we have seen in Remark \ref{rem:covering} with the 2-torus. Therefore, the introduction of Riemannian bounds truly allows to capture some symplectic rigidity, even when just considering Hausdorff-converging sequences. \par

However, the rigidity goes further than this. Indeed, the main motivation behind the study of Hausdorff-converging sequences is its importance when studying sequences converging in metrics coming from symplectic topology (c.f.\ \cite{Chasse2021}). Therefore, there is another rigidity question that crops up: does Theorem A of \cite{Chasse2021} holds for non-Lagrangian submanifolds? Of course, such a question makes no sense for most metrics coming from symplectic topology. One exception to this rule is however the Hofer pseudometric~\cite{Chekanov2000}, which makes sense for any submanifolds. \par

In other words, for a $n$-dimensional submanifold $N$ of a $2n$-dimensional symplectic manifold $M$, we can define\index{métrique de Hofer lagrangienne!généralisation non-lagrangienne}
\begin{align*}
d_H(N,N'):=\inf\left\{||\phi||_H\ \middle|\ \phi\in\mathrm{Ham}(M),\ \phi(N)=N'\right\},
\end{align*}
whenever $N'$ is Hamiltonian isotopic to $N$. Here, $||\cdot||_H$ denotes the Hofer norm. Let $\{N_i\}$ is a sequence of non-Lagrangian submanifolds converging in $d_H$ to $N_0$ and such that the inclusion $N_n\hookrightarrow M$ is in $\mathscr{I}_n(\Lambda,V)$ for some $\Lambda$ and $V$. Does $\{N_i\}$ behave like in the Lagrangian case, i.e.\ must $N_i\to N_0$ in $\delta_H$? \par

An obvious obstruction to that being the case is if $d_H$ is degenerate on the Hamiltonian orbit of $N$. However, by work of Usher~\cite{Usher2014}, this is precisely the case whenever $N$ is non-Lagrangian. In fact, $d_H(N,\cdot)\equiv 0$ whenever $N$ is nowhere Lagrangian, i.e.\ $\omega|_{T_xN}\neq 0$ for all $x\in N$. Furthermore, the set of nowhere Lagrangian embeddings $N\hookrightarrow M$ is residual in the $C^\infty$ topology and open in the $C^2$ topology whenever $n\geq 2$, i.e.\ whenever there are non-Lagrangian $n$-dimensional submanifolds. Therefore, $d_H$ is generically very much degenerate in the non-Lagrangian case. \par

This thus shows that at every step of the process, introducing Riemannian bounds does not reduce the Lagrangian case to the general one, but rather shows some new type of symplectic rigidity. \par

\begin{rem}
	The submanifold $N'$ such that $d_H(N,N')=0$ that we find is in $\mathscr{I}_n(\Lambda',V')$ for some $\Lambda'\geq\Lambda$ and $V'\geq V$, but not necessarily in $\mathscr{I}_n(\Lambda,V)$. This flexibility in the choice of a constant is however necessary to study symplectic~---~and not Riemannian~---~rigidity phenomena. For example, $\mathscr{I}_n(0,V)$ is made out of totally geodesic submanifolds, and we should expect some very strong rigidity, whether $N$ is Lagrangian or not.
\end{rem}

\section{Hausdorff limits of sequences of certain submanifolds} \label{sec:compactness}
In this section, we prove rigidity results for sequences of certain submanifolds of symplectic and contact manifolds. These results are shown mostly in the presence of Riemannian bounds, but some still hold without their presence. We also use this opportunity to prove Corollary \ref{cor:hausdorff_to_hofer} and relate it to the author's previous work. \par

\subsection{Proof of Theorem \ref{thm:eliashberg-gromov}} \label{subsec:compactness}
In this subsection, we prove the various parts of Theorem \ref{thm:eliashberg-gromov}. In order to make the presentation smoother, we however instead present it as a series of simpler results. \par

\begin{lem} \label{lem:compact_iso}
	Let $\{f_i:L_i\looparrowright M\}\subseteq \mathscr{I}_k(\Lambda,V)$ be a sequence of isotropic immersions of a symplectic manifold $(M,\omega)$ or of a contact manifold $(M,\xi)$. Suppose that $\{f_i(L_i)\}$ Hausdorff-converges to a closed subset $N$. Then, $N$ is the image of a $k$-dimensional isotropic $C^{1,\alpha}$-immersion $f_0:L_0\looparrowright M$, where $L_0$ is closed and connected.
\end{lem}

Note that we recover the Lagrangian case when $M$ is symplectic and $k=\frac{1}{2}\dim M$, and the Legendrian case when $M$ is contact and $k=\frac{1}{2}(\dim M-1)$

\begin{pr*}
	Suppose that $M$ is symplectic. Passing to a subsequence, we have diffeomorphisms $\phi_i:L_0\xrightarrow{\sim} L_i$ such that $\{f_i\circ\phi_i\}$ is $C^{1,\alpha'}$-converging to a $C^{1,\alpha}$-immersion $f_0:L_0\looparrowright M$, where $L_0$ is closed and connected. Since
	\begin{align*}
	\delta_H(f_i(L_i),f_0(L_0))=\delta_H(f_i(\phi_i(L_0)),f_0(L_0))\leq d_{C^0}(f_i\circ\phi_i,f_0),
	\end{align*}
	the sequence $\{f_i(L_i)\}$ must also converge to $\overline{f_0(L_0)}$ in the Hausdorff metric~---~recall that the Hausdorff metric is only truly a metric between closed subsets. Therefore, we must have $f_0(L_0)=N$ since $f_0(L_0)$ is compact, and thus closed. Finally, $0=f_i^*\omega$ converges in the $C^{0,\alpha'}$-topology to $f_0^*\omega$. Therefore, $f_0^*\omega=0$, and $f_0$ is isotropic. \par
	
	Suppose now that $M$ is contact with contact form $\alpha$. Then, the proof is analogous to the symplectic case: it suffices to replace $\omega$ by $\alpha$ in the proof above. If $M$ does not have a contact form, i.e.\ if $\xi$ is not coorientable, every point $p$ still has a neighborhood $U_p$ onto which $\xi=\Ker\alpha$. Since $\{f_i\circ\phi_i\}$ and its first order derivatives uniformly converge to $f_0$ on all compact subsets of $f_0^{-1}(U_p)$, the same argument still works.
\end{pr*}

\begin{lem} \label{lem:compact_coiso}
	Let $\{f_i:L_i\looparrowright M\}\subseteq \mathscr{I}_{n+k}(\Lambda,V)$ be a sequence of coisotropic immersions of a $2n$-dimensional symplectic manifold $(M,\omega)$ or of a co-oriented $(2n+1)$-dimensional contact manifold $(M,\xi=\Ker\alpha)$. Suppose that $\{f_i(L_i)\}$ Hausdorff-converges to a closed subset $N$. Then, $N$ is the image of a $(n+k)$-dimensional coisotropic $C^{1,\alpha}$-immersion $f_0:L_0\looparrowright M$, where $L_0$ is closed and connected.
\end{lem}

We recall that $f:L\looparrowright M$ is (symplectic) coisotropic if $(f_*(T_xL))^\omega\subseteq f_*(T_xL)$ for all $x\in L$, where 
\begin{align*}
V^\omega:=\{w\in T_yM\ |\ \omega_y(w,v)=0,\ \forall v\in V\}
\end{align*}
is the symplectic complement of a vector space $V\subseteq T_yM$. Following Huang~\cite{Huang2015}, we then say that $f:L\looparrowright (M,\xi=\Ker\alpha)$ is (contact) coisotropic\index{sous-variété contact coisotrope} if
\begin{align*}
\left(f_*(T_xL)\cap\xi_{f(x)}\right)^{d\alpha}\subseteq f_*(T_xL)\cap\xi_{f(x)}
\end{align*}
for all $x\in L$. Note that this definition depends only on $\xi$, not on the precise contact form $\alpha$ chosen. \par

\begin{pr*}
	Suppose that $M$ is symplectic. As in Lemma \ref{lem:compact_iso}, we have $\{\phi_i\}$ and $f_0=\lim_{C^{1,\alpha'}} (f_i\circ\phi_i)$ such that $f_0(L_0)=N$. Note that $f_i$ being coisotropic is equivalent to $\sigma_i:=(f_i\circ\phi_i)^*\omega$ having kernel
	\begin{align*}
	\Ker\sigma_{i,x}:=\{v\in T_xL_0\ |\ \sigma_{i,x}(v,w)=0,\ \forall w\in T_xL_0\}
	\end{align*}
	of dimension $n-k$ for all $x\in L_0$. As before, we have $C^{0,\alpha'}$-convergence of $\{\sigma_i\}$ to the $C^{0,\alpha}$-form $\sigma_0=f_0^*\omega$. Since the rank of matrices is lower semicontinuous, we must have $\dim\Ker\sigma_{0,x}\geq\dim\Ker\sigma_{i,x}=n-k$ for all $x\in L_0$. However, since $f_0$ is an immersion and $\omega$ is nondegenerate, $\dim\Ker\sigma_{0,x}\leq n-k$. Therefore, $f_0$ is coisotropic. \par
	
	When $M$ is contact, the proof is analogous, but the condition is instead that $\Ker\sigma_{i,x}|_{\xi_{f_i(\phi_i(x))}}$ has dimension $n-k$ for all $x\in L$. Since $\{f_i\circ\phi_i\}$ converges to $f_0$ in the $C^{1,\alpha'}$ topology, the $2n$-plane $\xi_{f_i(\phi_i(x))}$ converges to $\xi_{f_0(x)}$. Therefore, the proof goes through as before. 
\end{pr*}

\begin{rem} \label{rem:noncompact_symp}
	The proof of Lemma \ref{lem:compact_coiso} showcases well why the equivalent statement for symplectic submanifolds~---~or contact submanifolds~---~does not hold: the limit immersion might develop some degeneracy. For example, a generic nonsymplectic perturbation of the zero section  of $T^*L$ will be symplectic, even though the zero section itself is of course Lagrangian.
\end{rem}

We now go back to the symplectic isotropic case and show some additional rigidity when additional conditions are imposed on the immersions.

\begin{lem} \label{lem:compact_exact}
	If the $f_i$'s of Lemma \ref{lem:compact_iso} are exact for some Liouville form $\lambda$ on a symplectic manifold $(M,\omega=d\lambda)$, then the immersion $f_0:L_0\looparrowright M$ may be chosen so that there is a $C^{1,\alpha}$ function $h_0:L_0\to\R$ with $f_0^*\lambda=dh_0$.
\end{lem}
\begin{pr*}
	By hypothesis, each $L_i$ has a function $h_i:L_i\to\R$ such that $f_i^*\lambda=dh_i$. These functions are unique up to a constant. Therefore, we may fix $x\in L_0$ and take $h_i$ such that $h_i(\phi_i(x))=0$, where the $\phi_i$'s are the diffeomorphisms of Shen's theorem. Take $\bar{f}_i:=f_i\circ\phi_i$ and $\bar{h}_i:=h_i\circ\phi_i$, so that $(\bar{f}_i)^*\lambda=dh'_i$. In particular, the first order derivatives of $\bar{h}_i$ are uniformly $C^{0,\alpha'}$-bounded. Furthermore, for any $y\in L_0$, we have that
	\begin{align*}
	\big|\bar{h}_i(y)\big|
	&= \big|\bar{h}_i(y)-\bar{h}_i(x)\big| \\
	&= \left| \int_0^{d_{L_0}(x,y)}(d\bar{h}_i)_{\gamma(t)}\left(\dot{\gamma}(t)\right)dt\right| \\
	&\leq \big|\big|d\bar{h}_i\big|\big| \int_0^{d_{L_0}(x,y)}|\dot{\gamma}(t)|dt \\
	&\leq \mathrm{Diam}(L_0)\big|\big|d\bar{h}_i\big|\big|,
	\end{align*}
	where $\gamma$ is a unit-speed minimizing geodesic segment from $x$ to $y$, and $||\cdot||$ denotes here the supremum over $L_0$ of the operator norm. Since $||d\bar{h}_i||$ is uniformly bounded, the image of all $\bar{h}_i$'s is contained in some compact interval $I$. Therefore, $\{h_i\}$ is contained in a finite closed ball in $C^{1,\alpha'}(L_0,I)$ for $\alpha'<\alpha$. By compactness of this ball in the $C^{1,\alpha'}$ topology, we may pass to another subsequence which $C^{1,\alpha'}$-converges to a $C^{1,\alpha}$-function $h_0:L_0\to\R$. Taking the limit of $(\bar{h}_i)^*\lambda=d\bar{h}_i$ on both sides, we have the relation $f_0^*\lambda=dh_0$.
\end{pr*}

This proof shows the importance of working with H\"older spaces: if we only had uniform $C^0$-bounds on $d\bar{h}_i$, then we would only know that $h_0$ is continuous, and that $\{\bar{h}_i\}$ uniformly converges to $h_0$. In particular, the relation $f_0^*\lambda=dh_0$ would not necessarily hold. \par

We now turn our attention to sequences of weakly exact or monotone Lagrangian submanifolds. Contrary to what preceded, these results employ results from Section \ref{sec:sequences} in an essential way. \par

\begin{prop} \label{prop:compact_monotone}
	If the $f_i$'s of Lemma \ref{lem:compact_iso} are weakly exact Lagrangian embeddings or monotone Lagrangian embeddings with uniform monotonicity constant $\rho>0$, then so is $f_0$ whenever it is an embedding.
\end{prop}
\begin{pr*}
	For ease of notation, we will identify $f_i(L_i)$ with $L_i$, $i\geq 0$, and see $f_i$ simply as an inclusion. Suppose that the $L_i$'s are weakly exact. By Proposition \ref{prop:reparam}, there are finite coverings $p_i:L_i\to L_0$ such that $f_i$ may be $C^{1,\alpha'}$-approximated by $f_0\circ p_i$. In fact, by the proof of the proposition, we may pass to a subsequence so that the isomorphism type of $p_i$ is constant. Let $u:\D\to M$ be a disk with boundary along $L_0$ and symplectic area $\omega(u)$. \par
	
	Suppose that $\gamma:=u|_{\del\D=S^1}$ admits a lift $\tilde{\gamma}_i$ to $L_i$~---~since the isomorphism type of the covering is constant, this is independent of $i$. Let $v_i:S^1\times [0,1]\to M$ be a cylinder such $v_i|_{S^1\times\{0\}}=\tilde{\gamma}_i$ and $v_i|_{S^1\times\{1\}}=\gamma$ which is contained in the tubular neighborhood about $L_0$ of size $s(L_i;L_0)$. For example, we could take $v_i(t,s)=(1-s)\tilde{\gamma}_i(t)$ in a Weinstein neighborhood of $L_0$. But then, the concatenation $v_i\#u$ is a disk with boundary in $L_i$, which implies that
	\begin{align*}
	0=\omega(v_i\#u)=\omega(v_i)+\omega(u).
	\end{align*}
	However, we have that $\omega(v_i)\to 0$. To see this, we could for example equip $M$ with a metric which corresponds with the Sasaki metric on the previously-mentioned Weinstein neighborhood of $L_0$. Then, we get $\omega(v_i)\leq\mathrm{Area}(v_i)$, which obviously tends to 0. Therefore, we must have $\omega(u)=0$.
	
	If $\gamma$ does not admit a lift, there is some $k\geq 2$ such that $\gamma^k=u^k|_{S^1}$ does. Indeed, $p_i:L_i\to L_0$ is a finite covering, and thus $p_*(\pi_1(L_i))$ has finite index. But then,
	\begin{align*}
	0=\omega(u^k)=k\omega(u),
	\end{align*}
	which gives the result. \par
	
	Suppose now the $L_i$'s are monotone with uniform monotonicity constant $\rho>0$, i.e.\ $\omega=\rho\mu$, where $\mu:\pi_2(M,L_i)\to\Z$ is the Maslov index of $L_i$. The proof then goes through similarly as before. Indeed, when there is a lift $\tilde{\gamma}_i$, then we must have $\mu(v_i\# u)=\mu(u)$. This is because $\mu(v)$ depends only on the homotopy class of the path $t\mapsto T_{v(e^{it})}L\subseteq\R^{2n}$ in the Lagrangian Grassmannian. However, $v_i$ gives precisely a homotopy from the path associated to $u$ to the one associated to $v_i\# u$. Therefore, we have that
	\begin{align*}
	\rho\mu(u)=\rho\mu(v_i\# u)=\omega(v_i\# u)=\omega(v_i)+\omega(u).
	\end{align*}
	This again gives the result since $\omega(v_i)\to 0$. When there is no lift, the result also follows similarly as before:
	\begin{align*}
	k\rho\mu(u)=\rho\mu(u^k)=\omega(u^k)=k\omega(u),
	\end{align*}
	and $k>0$.
\end{pr*}

Note that we have everything we needed to prove Theorem~\ref{thm:eliashberg-gromov} and Corollary~\ref{cor:hausdorff_to_hofer}. \par

\begin{proof}[Proof of Theorem~\ref{thm:eliashberg-gromov}]
	The statement on isotropic~---~and thus Lagrangian and Legendrian~---~immersions is Lemma~\ref{lem:compact_iso}, and the one on coisotropic immersions is Lemma~\ref{lem:compact_coiso}. The statement on exact Lagrangian immersions is Lemma~\ref{lem:compact_exact}, and the one on weakly exact and monotone Lagrangian embeddings is Proposition~\ref{prop:compact_monotone}.
\end{proof}

\begin{proof}[Proof of Corollary~\ref{cor:hausdorff_to_hofer}]
	Using a Weinstein neighborhood, we may assume without loss of generality that $L_i$ is an exact Lagrangian submanifold in $T^*L_0$, and the $f_i$'s are simply inclusions. Then, $L_i$ Hausdorff-converges to the zero section, and $f_0$ is the natural inclusion $L_0\hookrightarrow T^*L_0$. By Theorem \ref{thm:section}, $L_i$ is the graph of some exact 1-form $dh_i$ for $i$ large enough. \par
	
	Consider $H_i:=\beta \pi^*h_i$, where $\pi:T^*L_0\to L_0$ is the natural projection, and $\beta$ is a bump function equal to 1 in some codisk bundle $D^*_rL_0$ containing all $L_i$ for $i$ large and equal to 0 outside some other codisk bundle $D^*_RL_0$. Then, the Hamiltonian diffeomorphism that it generates sends $L_0$ to $L_i$. Therefore,
	\begin{align*}
	d_H(L_0,L_i)\leq \max_{x\in T^*L_0} H_i(x)-\min_{y\in T^*L_0} H_i(y)\leq \max_{x\in L_0} h_i(x)-\min_{y\in L_0} h_i(y)
	\end{align*}
	by the definition of the Lagrangian Hofer metric. By compactness of $L_0$, there are points $x_i,y_i\in L_0$ where $h_i$ attains its maximum and minimum, respectively. We may then take a unit-speed minimizing geodesic $\gamma_i$ from $x$ to $y$. Then,
	\begin{align*}
	\max_{x\in L_0} h_i(x)-\min_{y\in L_0} h_i(y)
	&= \left|\int_0^{d_{L_0}(x_i,y_i)}(dh_i)_{\gamma_i(t)}(\dot{\gamma}_i(t))dt\right| \\
	&\leq ||dh_i|| \int_0^{d_{L_0}(x_i,y_i)}|\dot{\gamma}_i(t)|dt \\
	&\leq \mathrm{Diam}(L_0)||dh_i||.
	\end{align*}
	However, $||dh_i||=s(L_i;L_0)$ when $T^*L_0$ is equipped with the Sasaki metric, because $L_i=\operatorname{graph}dh_i$. Since convergence in the Hausdorff metric is independent on the distance function, and since we know that $\delta_H(L_i,L_0)\to 0$ in some distance function, then $||dh_i||\to 0$. Therefore, $d_H(L_0,L_i)\to 0$.
\end{proof}

\begin{rem} \label{rem:hausdorff_to_hofer}
	\hspace{2em}
	\begin{enumerate}[label=(\arabic*)]
		\item Following Remark \ref{rem:section-vs-covering}, there is an analogous statement for immersions if we instead consider $L_0$ in its normal bundle in $f_0^*TM$. Then, a neighborhood of $L_0$ can be identified with a neighborhood of the zero section of $T^*L_0$ using an $\omega$-compatible almost complex structure.
		\item In light of the rigidity of the Hofer metric for coisotropic submanifolds proved by Usher~\cite{Usher2014}, we expect that a similar result also holds for them under adequate conditions. Said conditions are however unclear for the time being.
		\item Likewise, we expect a similar result for Legendrian submanifolds~---~or more generally contact coisotropic submanifolds~---~with the Shelukhin-Hofer pseudometric\index{pseudométrique de Shelukhin-Hofer} as defined by Rosen and Zhang~\cite{RosenZhang2018}, based on the metric on contactomorphisms of Shelukhin~\cite{Shelukhin2017}.
		\item Note that that Lemma \ref{cor:hausdorff_to_hofer} implies that $\{f_i(L_i)\}$ also converges to $L_0$ in the spectral metric. However, we could have gotten this result directly from Corollary \ref{cor:conjectures} via a rescaling argument \textit{\`a la} Shelukhin~\cite{Shelukhin2018}.
	\end{enumerate}
\end{rem}

\subsection{The tame and bounded volume conditions} \label{subsec:tame+volume}
In this subsection, we explore the relation between the $\epsilon$-tameness condition of the author's previous work~\cite{Chasse2021} and the condition of having bounded volume. \par

We recall that a submanifold $N$ of a Riemannian manifold $(M,g)$ is said to be $\epsilon$-tame\index{$\epsilon$-tame}\index{$\epsilon$-dominé}, $0<\epsilon\leq 1$, if
\begin{align*}
\frac{d_M(x,y)}{\min\{1,d_N(x,y)\}}\geq\epsilon
\end{align*}
for all $x\neq y\in N$. Here, $d_M$ denotes the metric on $M$ induced by the Riemannian metric $g$, whilst $d_N$ denotes the metric on $N$ induced by the restriction $g|_{TN}$ of $g$ to $N$. \par

\begin{prop} \label{prop:bounded_volume}
	Take $\Lambda\geq 0$ and $\epsilon\in (0,1]$. Let $K\subseteq M$ be compact. There exists a constant $V=V(\Lambda,\epsilon,k,K)>0$ such that whenever $N$ is a (closed connected) $\epsilon$-tame $k$-dimensional submanifold in $K$ with $||B_N||\leq\Lambda$, then the inclusion $N\hookrightarrow M$ is in $\mathscr{I}_k(\Lambda,V)$.
\end{prop}

We leave the proof of the proposition for later and give an application of the result when combined with Theorem \ref{thm:eliashberg-gromov}. In order to do this, we recall the notion central to the author's previous work~\cite{Chasse2021}: Chekanov-type metrics. Broadly speaking, these are metrics $\hat{d}^{\mathscr{F},\mathscr{F'}}$ defined on collections of Lagrangian submanifolds using auxiliary families $\mathscr{F}$ and $\mathscr{F'}$ and behaving well with regards to $J$-holomorphic curves. More precisely, they are defined by the property that for any compatible almost complex struction $J$, any pair of Lagrangian submanifolds $L$, $L'$ in the collection, we have~---~up to arbitrarily small Hamiltonian perturbations~---~for any $x\in L\cup L'$,
\begin{align*}
\hat{d}^{\mathscr{F},\mathscr{F'}}(L,L')\leq \omega(u)
\end{align*}
for some $J$-holomorphic curve $u$ with boundary along $L$, $L'$ and elements of $\mathscr{F}$ or $\mathscr{F'}$, and which passes through $x$. Most notably, examples of Chekanov-type metric include the Lagrangian Hofer metric (with $\mathscr{F}=\mathscr{F'}=\emptyset$), the spectral metric (also with $\mathscr{F}=\mathscr{F'}=\emptyset$), and the shadows metrics. \par

\begin{cor} \label{cor:equiv_topology}
	Let $\hat{d}^{\mathscr{F},\mathscr{F'}}$ be a Chekanov-type metric which is bounded from above by the Lagrangian Hofer metric $d_H$. Then, for any compact $K\subseteq M$, $\hat{d}^{\mathscr{F},\mathscr{F'}}$ induces the same topology on $\mathscr{L}^e_{\Lambda,\epsilon}(K)$ as the Hausdorff metric. \par
	
	If $V>0$, then the same result holds on the subset of $\mathscr{L}^e_{\Lambda,\epsilon}(M)$ composed of Lagrangian submanifolds having volume bounded from above by $V$, whether $M$ is compact or not.
\end{cor}

We refer to the author's previous work for the precise definition of what a Chekanov-type metric is. Note however that the Lagrangian Hofer metric, the spectral metric, and all shadow metrics are of Chekanov type. Here, $\mathscr{L}^e_{\Lambda,\epsilon}(K)$ denotes the collection of all $\epsilon$-tame exact Lagrangian submanifolds $L$ of $M$ contained in $K$ and such that $||B_L||\leq\Lambda$. \par

\begin{pr*}
	By the author's previous work~\cite{Chasse2021}, every $\hat{d}^{\mathscr{F},\mathscr{F'}}$-converging sequence in $\mathscr{L}^e_{\Lambda,\epsilon}(M)$ also converges in the Hausdorff metric to the same limit. If we are given a volume bound $V>0$, then every Hausdorff-converging sequence in the associated subset of $\mathscr{L}^e_{\Lambda,\epsilon}(M)$ also converges in $d_H$ to the same limit by Theorem \ref{thm:eliashberg-gromov}. On $\mathscr{L}^e_{\Lambda,\epsilon}(K)$, we automatically a volume bound by Proposition \ref{prop:bounded_volume}. Since $\hat{d}^{\mathscr{F},\mathscr{F'}}\leq d_H$ by hypothesis, this gives the result. 
\end{pr*}

\begin{rem} \label{rem:equiv_topology}
	Note that Corollary \ref{cor:equiv_topology} implies that $d_H$ is bounded on $\mathscr{L}^e_{\Lambda,\epsilon}(D^*L)$. This is in stark contrast with the behavior of $d_H$ on $\mathscr{L}^e(D^*L)$, i.e.\ without any Riemannian bounds, where it is expected to be unbounded~\cite{Shelukhin2018}. Therefore, Hamiltonian diffeomorphisms which moves a Lagrangian submanifold a lot in the Lagrangian Hofer metric must also greatly deform it. \par
	
	Note that $\gamma$ is however expected to be bounded on $\mathscr{L}^e(D^*L)$~---~that is precisely the conjecture of Viterbo. It also follows from work of Biran and Cornea~\cite{BiranCornea2021} that some fragmentation metrics are bounded on $\mathscr{L}^e(D^*L)$. It may thus be that $\mathscr{L}^e_{\Lambda,\epsilon}(D^*L)$ better capture the topology in these metrics than in the Lagrangian Hofer metric.
\end{rem}

We now give the proof of Proposition \ref{prop:bounded_volume}; it relies mostly on the Bishop--Gromov inequalities. \par

\begin{proof}[Proof of Proposition \ref{prop:bounded_volume}]
	Note first that it suffices to bound the diameter of $N$. Indeed, the bound $\Lambda$ on the second fundamental form, together with Gauss' equation, gives an upper bound $\lambda=\lambda(\Lambda,K)\geq 0$ on the absolute value of the sectional curvature of $N$. Therefore, by the Bishop--Gromov inequality, we have that
	\begin{align*}
	\mathrm{Vol}(B^N_r(x))\leq \mathrm{Vol}(B^{M^k(-\lambda)}_r(x')),
	\end{align*}
	for any $x\in N$, any $x'\in M^k(-\lambda)$, and any $r>0$. Here, $M^k(-\lambda)$ is the $k$-dimensional simply-connected space of constant sectional curvature $-\lambda$. In particular, we get
	\begin{align*}
	\mathrm{Vol}(N)
	&\leq \mathrm{Vol}\left(B^{M^k(-\lambda)}_{\mathrm{Diam}(N)}(x')\right) = \frac{2\pi^{\frac{k}{2}}}{\Gamma(\frac{k}{2})}\int_0^{\mathrm{Diam}(N)}\left(\frac{\sinh(t\sqrt{\lambda})}{\sqrt{\lambda}}\right)^{k-1}dt.
	\end{align*}
	When $\lambda=0$, the quotient $\sinh(t\sqrt{\lambda})/\sqrt{\lambda}$ should be interpreted as being equal to $t$. Since $\sinh t$ (or $t$) is increasing and nonnegative, an upper bound on $\mathrm{Diam}(N)$ will thus indeed give an upper bound on $\mathrm{Vol}(N)$. \par
	
	We now bound the diameter of $N$. Note that by Shen's work~\cite{Shen1995}, there exists $r_0=r_0(\Lambda,K)>0$ such that the injectivity radius $r_\mathrm{inj}(N)$ of $N$ respects $r_\mathrm{inj}(N)\geq r_0$. Furthermore, since $N$ is closed and connected, there exist $x,y\in N$ such that $d(x,y)=\mathrm{Diam}(N)=:T$ and a unit-speed minimizing geodesic $\gamma$ of $N$ such that $\gamma(0)=x$ and $\gamma(T)=y$. Set $x_i:=\gamma(i)$ for $i\in\{0,1,\dots,\lfloor T\rfloor\}$. By the construction, we have that $d_N(x_i,x_j)\geq 1$ if $i\neq j$. Therefore, $d_M(x_i,x_j)\geq\epsilon$ by the tameness condition, i.e.\
	\begin{align*}
	\bigsqcup_{i=1}^{\lfloor T\rfloor}B^M_\epsilon(x_i) &\subseteq B^M_\epsilon(K),
	\intertext{and thus}
	\sum_{i=1}^{\lfloor T\rfloor}\mathrm{Vol}\left(B^M_\epsilon(x_i)\right) &\leq \mathrm{Vol}\left(B^M_\epsilon(K)\right).
	\end{align*}
	Taking $r:=\min\{r_0,\epsilon\}$, and using the other side of the volume comparison theorem, we get
	\begin{align*}
	\mathrm{Vol}\left(B^M_\epsilon(K)\right)
	&\geq \sum_{i=1}^{\lfloor T\rfloor}\mathrm{Vol}\left(B^M_\epsilon(x_i)\right) \\
	&\geq \sum_{i=1}^{\lfloor T\rfloor}\mathrm{Vol}\left(B^M_r(x_i)\right) \\
	&\geq \sum_{i=1}^{\lfloor T\rfloor}\mathrm{Vol}\left(B^{M^k(\lambda)}_r(p')\right) \\
	&= \frac{2\pi^{\frac{k}{2}}\lfloor T\rfloor}{\Gamma(\frac{k}{2})}\int_0^r\left(\frac{\sin(t\sqrt{\lambda})}{\sqrt{\lambda}}\right)^{k-1}dt.\\
	\end{align*}
	Since $\lfloor T\rfloor\geq T-1$, this does give an upper bound on $T=\mathrm{Diam}(N)$.
\end{proof}

\begin{rem} \label{rem:bound_volume}
	\hspace{2em}
	\begin{enumerate}[label=(\arabic*)]
		\item When $M$ is itself compact, then the dependence of the bound on $K$ may be replaced by a dependence on the volume of $M$ and on bounds on its sectional curvature and injectivity radius.
		\item The bound is sharp when $\lambda=0$, $r=\epsilon$ and $T$ is an integer, e.g.\ when $N$ is a curve or an isometrically embedded flat torus with integer diameter. 
		\item A slicker proof exists given the existence of a volume comparison theorem for tubes about submanifolds, as we would then simply have
		\begin{align*}
		\mathrm{Vol}\left(B^M_r(K)\right)\geq \mathrm{Vol}\left(B^M_r(N)\right)\geq C(r,\Lambda,\epsilon,k)\mathrm{Vol}(N)
		\end{align*}
		for any $r\in(0,s(C(N);N))$, where $C(N)$ is the cut locus of $N$. Indeed, as noted by Groman and Solomon~\cite{GromanSolomon2014}, the tameness condition allows us to estimate $s(C(N);N)$. However, such a comparison theorem in full generality seems beyond the reach of current technology, although some particular cases are known (cf.\ \cite{Gray2004}).
	\end{enumerate}
\end{rem}

\subsection{Limits in the absence of volume bounds} \label{subsec:rigidity_without_volume_bounds}
In this subsection we explain how a lot of the compactness results that we have presented in Subsection \ref{subsec:compactness} still holds when $V=\infty$, i.e.\ without the presence of volume bounds. In what follows, $\mathscr{I}'_k(\Lambda,\infty;z_0)$ will denote the space of smooth pointed immersions $f:(N,x_0)\looparrowright (M,z_0)$ such that $N$ is a~---~possibly noncompact~---~connected $k$-dimensional manifold without boundary, and $||B_f||\leq\Lambda$. Furthermore, if $N$ is noncompact and $g$ is the Riemannian metric of $M$, we then ask that $(N,f^*g)$ be complete. \par

The techniques that we will use are based on a pointed version of Shen's theorem, which also follows Shen's work. This is because Shen's theorem uses estimates valid on any complete Riemannian manifold and Kasue's version of Gromov's compactness theorem~\cite{Kasue1989}, of which there exists a pointed version. This is also a special case of Theorem~1.2 of \cite{Smith2007}. For ease of presentation, we give an explicit statement. \par

\begin{thm*}[\cite{Shen1995,Smith2007}]
	Let $\{f_i:(N_i,x_i)\looparrowright (M,z_0)\}_{i\geq 1}\subseteq \mathscr{I}'_k(\Lambda,\infty;z_0)$. Let $0<\alpha'<\alpha<1$. Then, we have the following:
	\begin{enumerate}[label=(\roman*)]
		\item a subsequence, still denoted $\{f_i\}$;
		\item a connected $k$-dimensional pointed smooth manifold without boundary $(N_0,x_0)$;
		\item a complete Riemannian metric $g_0$ on $N_0$ of class $C^{1,\alpha}_\mathrm{loc}$;
		\item an immersion $f_0:(N_0,x_0)\looparrowright (M,z_0)$ of class $C^{1,\alpha}_\mathrm{loc}$ with $f_0^*g=g_0$;
		\item for each $i$, a map $\phi_i:(N_0,x_0)\to (N_i,x_i)$ of class $C^{2,\alpha}_\mathrm{loc}$
	\end{enumerate}
	such that for all compact neighborhood $K$ of $z_0$,
	\begin{enumerate}[label=(\Roman*)]
		\item $\phi_i|_K$ is a diffeomorphism onto its image for $i$ large enough;
		\item $\{f_i\circ\phi_i|_K\}$ converges in the $C^{1,\alpha'}$ topology to $f_0|_K$.
	\end{enumerate}
\end{thm*}

Using this theorem, we can now prove the following generalization of some results of Subsection \ref{subsec:compactness}. \par

\begin{cor} \label{cor:compactness_without_volume_bounds}
	The results of Lemmata \ref{lem:compact_iso} and \ref{lem:compact_coiso} still holds, with $\overline{f_0(L_0)}=N$ instead, even if $V=\infty$ and the $L_i$'s are noncompact. The result of Lemma \ref{lem:compact_exact} also still holds, but $h_0$ is instead in $C^{1,\alpha}_\mathrm{loc}$.
\end{cor}

For the proof, we will need a small technical construction, whose proof we leave to the end of this subsection. \par

\begin{prop} \label{prop:sequence_ham}
	Let $M$ be a symplectic manifold, and take a sequence $\{x_i\}\subseteq M$ converging to $x_0\in M$. Then, there exists a sequence of Hamiltonian diffeomorphisms with compact support $\{\psi_i\}$ converging in the $C^2$-topology to the identity and such that $\psi_i(x_i)=x_0$. There is an analogous result for contactomorphisms when $M$ is a contact manifold.
\end{prop}

\begin{proof}[Proof of Corollary \ref{cor:compactness_without_volume_bounds}]
	Suppose that $\{f_i:L_i\looparrowright M\}$ is a sequence of isotropic or coisotropic immersions such that $\{f_i(L_i)\}$ Hausdorff-converges to $N$. Take $z_0\in N$. By Hausdorff-convergence, for each $i$, there is $y_i\in f_i(L_i)$ such that $\lim y_i=z_0$. By Proposition \ref{prop:sequence_ham}, there is a sequence of symplectomorphisms or contactomorphisms $\{\psi_i\}$ which $C^2$-converges to the identity and such that $\psi_i(y_i)=z_0$. \par
	
	We take $x_i\in f_i^{-1}(y_i)$ and $f'_i:=\psi_i\circ f_i$. Then, $C^2$-convergence insures that $\{f'_i:(L_i,x_i)\to (M,z_0)\}\subseteq\mathscr{I}'_k(\Lambda',\infty;z_0)$ for some $\Lambda'\geq\Lambda$. The rest of the proofs of the lemmata then goes through as before, except that we instead use the pointed version of Shen's theorem. This works because each point of $N_0$ is contained in \textit{some} compact neighborhood of $x_0$, and being (exact) (co)isotropic is a local condition about that point. We only get $\overline{f_0(L_0)}=N$, because $f_0(L_0)$ might not be closed if $L_0$ is noncompact. 
\end{proof}

\begin{rem} \label{rem:monotone_without_volume_bounds}
	The proof of Proposition \ref{prop:compact_monotone} relies in an essential way on the fact that there is a covering $L_i\to L_0$, i.e.\ Proposition \ref{prop:reparam}. The proof of that relies on applying Ehresmann's fibration theorem to $p=\iota^{-1}\circ f_0:L_0\to L$, which requires $p$ to be proper. When $L_0$ is compact, this is of course automatic, but not in the noncompact case. For example, we could modify the example in Remark \ref{rem:covering} to get
	\begin{center}
		\begin{tikzcd}[row sep=0pt,column sep=1pc]
			f_i\colon \T^2 \arrow{r} & \T^2\times\R^2=T^*\T^2 \\
			{\hphantom{\alpha\colon{}}} (\theta_1,\theta_2) \arrow[mapsto]{r} & {\left(i\theta_1,\theta_2,\frac{1}{i}\cos\theta_1,\frac{1}{i}\sin\theta_1\right)}.
		\end{tikzcd}
	\end{center}
	This gives in the limit a map $f_0:\R\times S^1\to T^*\T^2$ such that the corresponding $p$ is not proper. Note that $f_0$ is nonetheless a covering. We expect that to still be the case whenever the $L_i$'s are closed manifold; one can however easily make counterexamples when they are not.
\end{rem}

\begin{proof}[Proof of Proposition \ref{prop:sequence_ham}]
	We begin with the case where $M$ is symplectic. Let $\phi:U\to B^{2n}_\delta(0)\subseteq\R^{2n}$ be a Darboux chart centered at $x_0$. Take a rotation-invariant bump function $\beta:\R^{2n}\to [0,1]$ with support in $B^{2n}_\delta(0)$ and such that $\beta|_{B^{2n}_{\delta'}(0)}\equiv 1$ for some $\delta'\in (0,\delta)$. We fix $u\in\R^{2n}$ with $|u|=1$, and consider the Hamiltonian on $\R^{2n}$
	\begin{align*}
	H(v):=\beta(v)\omega_0(v,u),
	\end{align*}
	where $\omega_0$ is the standard symplectic form on $\R^{2n}$. The Hamiltonian isotopy $\{\psi_t\}$ that it generates is such that $\psi_t(v)=v-tu$ whenever $|v-su|<\delta'$ for all $s\leq t$. \par
	
	Note that for $i$ large enough, not only is $x_i$ in $U$, but also $\phi(x_i)\in B^{2n}_{\delta'}(0)$. Suppose that we have such $i$. Let $R_i$ be a unitary transformation sending $\phi(x_i)$ to $|\phi(x_i)|u$, and define a Hamiltonian $H_i$ on $M$ by
	\begin{align*}
	H_i(x):=\begin{cases}
	|\phi(x_i)|H(R_i\phi(x)) \quad &\text{if $x\in U$}; \\
	0 \quad &\text{otherwise}.
	\end{cases}
	\end{align*}
	Let $\{\psi^i_t\}$ be the Hamiltonian isotopy that it generates. A direct computation gives that $\phi(\psi^i_t(x))=R_i^{-1}\psi_{|\psi(x_i)|t}(R_i\phi(x))$ whenever $x\in U$. In particular, $\psi^i_1(x_i)=x_0$. Since $\psi_i:=\psi^i_1$ is the identity outside $U$, it also follows from this relation that the sequence $\{\psi_i\}$ converges to $\Id_M$ in $C^2$-topology. \par
	
	The case when $M$ is contact is quite similar. Indeed, we can still take a Darboux chart $\phi$ centered at $x_0$, and consider the contact Hamiltonian
	\begin{align*}
	H(v):=\beta(v)(d\alpha_0)(v,u),
	\end{align*}
	where $\beta$ is a bump function with support in $B^{2n+1}_\delta(0)$, $\alpha_0$ is the standard contact form on $\R^{2n+1}$, and $u\in \R^{2n}\times\{0\}$ is unitary. The contact isotopy $\{\psi^H_t\}$ is quite similar to what we had in the symplectic case: if we write $u=(x^0_i,y^0_i,0)_{1\leq i\leq n}$, then
	\begin{align*}
	\psi^H_t(x_i,y_i,z)=\left(x_i-tx^0_i,y_i-ty^0_i,z+\sum_i \left(\frac{x^0_iy^0_i}{2}t-y^0_ix_i\right)t\right)
	\end{align*}
	whenever $v=(x_i,y_i,z)\in B^{2n+1}_{\delta'}(0)$. Therefore, for any $z^0\in\R$, we get that
	\begin{align*}
	\left(\psi^{\beta}_{z_0+\frac{1}{2}\sum_i x^0_iy^0_i}\circ\psi^H_1\right)(x^0_i,y^0_i,z^0)=0
	\end{align*}
	if $(x^0_i,y^0_i,z^0)\in B^{2n+1}_{\delta'}(0)$. The rest of the argument is then completely analogous to the symplectic case.
\end{proof}

\begin{rem} \label{rem:sequence_ham_hofer}
	The construction in the symplectic case actually gives a sequence which also converges to the identity in the Hofer norm. Indeed, in the symplectic case, we have that
	\begin{align*}
	||\psi_i||_H
	&\leq \int_0^1\left(\max_{x\in M}H_i(x)-\min_{x\in M}H_i(x)\right)dt \\
	&= |\phi(x_i)|\left(\max_{|v|\leq\delta}H(v)-\min_{|v|\leq\delta}H(v)\right) \\
	&\leq 2\delta|\phi(x_i)|
	\end{align*}
	which of course tends to 0. \par
	
	Likewise, if $M$ is a contact manifold admitting a contact form $\alpha$, then the same argument implies that $\{\psi_i\}$ converges to the identity in the Shelukhin-Hofer norm\index{métrique de Shelukhin-Hofer!norme de Shelukhin-Hofer} associated to $\alpha$~\cite{Shelukhin2017}.
\end{rem}

\subsection{Limits in the absence of any Riemannian bounds} \label{subsec:rigidity_without_bounds}
We now prove Theorem~\ref{thm:equiv_limits}, which does show that there is some rigidity for sequences of Lagrangian submanifolds, even when no Riemannian bounds are put on such sequence. This proves that there exists some rigidity for the Hausdorff metric between Lagrangian submanifolds in full generality. \par

Seeing this from the other way around, this implies that $\hat{d}^{\mathscr{F},\mathscr{F'}}$-converging sequences either behave like those which are geometrically bounded~---~although they may not themselves be geometrically bounded (c.f.\ \cite{Chasse2021})~---~or they Hausdorff-converge to a fairly pathological space. This is thus a good step in the direction of a truly symplectic characterization of ``nicely-behaved sequences'' of Lagrangian submanifolds. \par

Furthermore, this indicates that it makes sense to talk of $C^0$-Lagrangian submanifolds as $n$-dimensional topological submanifolds $L$ of $M$ which are the Hausdorff limit of a sequence of smooth Lagrangian submanifolds $\{L_i\}$ such that the sequence $\{L_i\}$ is also Cauchy in a nice Chekanov-type metric $\hat{d}^{\mathscr{F},\mathscr{F'}}$. Note that a sequence of maps $\{\psi_i\}$ $C^0$-converge to a map $\psi$ if and only if the sequence of graphs $\{\operatorname{graph}\psi_i\}$ Hausdorff-converges¸ to $\operatorname{graph}\psi$~\cite{Waterhouse1976}. Therefore, graphs of hameomorphisms~\cite{OhMuller2007} are $C^0$-Lagrangians in our sense for $\hat{d}^{\mathscr{F},\mathscr{F'}}=d_H$. However, Oh and M\"uller's definition is \textit{a priori} stronger than ours in the sense that there may be Lagrangian graphs of homeomorphisms which are not hameomorphisms. Instead, they would be graphs of homeomorphisms obtained as limits of Hamiltonian diffeomorphisms in what they call the weak Hamiltonian topology. On top of that, there could be graphs of $C^0$-limits of non-Hamiltonian symplectomorphisms when $\hat{d}^{\mathscr{F},\mathscr{F'}}$ can compare non-Hamiltonian diffeomorphic Lagrangian submanifolds. \par

Likewise, it is unclear how this definition of $C^0$-Lagrangian submanifolds relate to the definition of Humili\`ere, Leclercq, and Seyfaddini~\cite{HumiliereLeclercqSeyfaddini2015}. Indeed, our definition is global whilst theirs is local, which makes a direct comparison hard. Likewise, yet another definition of $C^0$-Lagrangian submanifold as been proposed very recently by Viterbo~\cite{Viterbo2022ii} using sheaves. \par

Finally, it was pointed out to us by Seyfaddini that this can be seen as a generalization of results of Hofer~\cite{Hofer1990} and Viterbo~\cite{Viterbo1992} saying that if a sequence of Hamiltonian diffeomorphisms $\{\psi_i\}$ is such that
\begin{enumerate}[label=(\arabic*)]
	\item $\psi_i\xrightarrow{C^0} \phi$;
	\item $\psi_i\xrightarrow{||\cdot||_H} \psi$ or $\psi_i\xrightarrow{\gamma} \psi$,
\end{enumerate}
then $\psi=\phi$. As noted before, (1) here is equivalent to (1) in Theorem \ref{thm:equiv_limits} with $L_i=\operatorname{graph}\psi_i$. However, the two versions of (2) are \textit{a priori} entirely independent. Indeed, the Hofer and spectral norms on Hamiltonian diffeomorphisms have notoriously different behavior than the Hofer and spectral norms on their graph (c.f.\ \cite{Ostrover2003}). These differences are however on the large scale geometry of the metrics; this indicates that the local behavior are similar. \par

The proof of Theorem~\ref{thm:equiv_limits} follows from the following lemma, which follows directly from the proof of Theorem~1 in the author's previous work~\cite{Chasse2021}. \par

\begin{lem} \label{lem:half_ineq}
	Let $\hat{d}^{\mathscr{F},\mathscr{F'}}$ be a Chekanov-type metric on a collection of Lagrangian submanifolds $\mathscr{L}^\star(M)$ (c.f.\ \cite{Chasse2021}). For every $L\in \mathscr{L}^\star(M)$, there exists $\delta>0$ and $C>0$ such that
	\begin{align*}
		Cs\left(L;L'\cup\left[\left(\bigcup_{F\in\mathscr{F}} F\right)\cap \left(\bigcup_{F'\in\mathscr{F'}} F'\right)\right]\right)\leq \hat{d}^{\mathscr{F},\mathscr{F'}}(L,L')
	\end{align*}
	for all $L'\in\mathscr{L}^\star(M)$ such that $\hat{d}^{\mathscr{F},\mathscr{F'}}(L,L')<\delta$, where we recall that
	\begin{align*}
		s(A;B):=\sup_{x\in A}d_M(x,B):=\sup_{x\in A}\inf_{y\in B}d_M(x,y).
	\end{align*}
	for any subsets $A,B\subseteq M$. In particular, if $\hat{d}^{\mathscr{F},\mathscr{F'}}=d_H$ or $\gamma$, we have that
	\begin{align*}
		Cs\left(L;L'\right)\leq \hat{d}^{\mathscr{F},\mathscr{F'}}(L,L').
	\end{align*}
\end{lem}

\begin{proof}[Proof of Theorem \ref{thm:equiv_limits}]
	First of all, note that $N$ is connected since all $L_i$ are. By Lemma~\ref{lem:half_ineq}, we have a constant $C=C(M,L_0)>0$ such that
	\begin{align*}
	Cs\left(L_0;L_i\cup\left[\left(\bigcup_{F\in\mathscr{F}} F\right)\cap \left(\bigcup_{F'\in\mathscr{F'}} F'\right)\right]\right)\leq \hat{d}^{\mathscr{F},\mathscr{F'}}(L_0,L_i)  \tag{$\star$}
	\end{align*}
	for $i$ large. Therefore, $L_0$ is contained in $N\cup((\cup F)\cap (\cup F'))$. Since $(\cup F)\cap (\cup F')$ is totally disconnected, $L_0$ must thus be contained in $N$. Since $L_0$ is compact, it is a closed subset of $N$. But since $L_0$ and $N$ are both submanifolds of $M$, we must have $m=n$, and $L_0$ must be an open subset of $N$. Therefore, by connectedness, $N=L_0$.
\end{proof}

\begin{rem} \label{rem:equiv_limits}
	Note that the technical upgrade of Sikorav's version of the monotonicty lemma~\cite{Sikorav1994} proven in the paper cited above is not necessary here. Since the constant $C$ only needs to depend on $L_0$~---~and not on metric bounds of $L_0$~---~Lemma~\ref{lem:half_ineq} also follows from applying Sikorav's monotonicity lemma on a small-enough metric ball centered at a point of $L_0$.
\end{rem}

\clearpage
\bibliographystyle{alpha}
\bibliography{sequences}

\textsc{Department of Mathematics, ETH Z\"urich HG, R\"amistrasse 101, 8092 Zurich, Switzerland} \\
\textit{E-mail}: \href{mailto:jeanphilippe.chasse@math.ethz.ch}{jeanphilippe.chasse@math.ethz.ch}
\end{document}